\documentclass{amsart}
\usepackage{amsfonts,amssymb,amsmath}

\newtheorem{theorem}{Theorem}
\newtheorem{acknowledgement}[theorem]{Acknowledgement}

\newtheorem{corollary}[theorem]{Corollary}

\newtheorem{lemma}[theorem]{Lemma}

\newtheorem{proposition}[theorem]{Proposition}

\begin{document}
\title{On stable range one matrices}
\author{Grigore C\u{a}lug\u{a}reanu, Horia F. Pop}

\begin{abstract}
For $2\times 2$ matrices over commutative rings, we prove a characterization
theorem for left stable range 1 elements, we show that the stable range 1
property is left-right symmetric (also) at element level, we show that all
matrices with one zero row (or zero column) over B\'{e}zout rings have
stable range 1. Using diagonal reduction, we characterize all the $2\times 2$
integral matrices which have stable range 1 and discuss additional
properties including Jacobson's Lemma for stable range 1 elements. Finally,
we give an example of exchange stable range 1 integral $2\times 2$ matrix
which is not clean.
\end{abstract}

\maketitle


\section{Introduction}

Recall that a (unital) ring $R$ has \textsl{(left) stable range 1} provided
that for any $a,b\in R$ satisfying $Ra+Rb=R$, there exists $y\in R$ such
that $a+yb$ is left invertible. This condition is left-right symmetric. In a
ring with stable range 1, all one-sided inverses are two-sided, and so in
the definition $a+yb$ must be a unit. Equivalently, $R$ has stable range 1
if for any $a,x,b\in R$ satisfying $xa+b=1$, there exists $y\in R$ such that 
$a+yb$ is a unit. For any positive integer $n$, the matrix ring $\mathbb{M}%
_{n}(R)$ has stable range 1 if and only if $R$ has stable range 1. 

It follows from the definition that stable range 1 rings have an adequate
supply of units. That's why, rings with only few units do not have this
property (e.g. $\mathbb{Z}$, the ring of the integers).

In the sequel $R$ denotes a unital ring, $U(R)$ denotes the set of all the
units of $R$ and $J(R)$ the Jacobson radical of $R$. By $E_{ij}$ we denote
the $n\times n$ matrix with all entries zero, excepting the $(i,j)$-entry,
which is $1$. Whenever it is more convenient, we will use the widely
accepted shorthand \textquotedblleft iff\textquotedblright\ for
\textquotedblleft if and only if\textquotedblright\ in the text.

In \cite{che} we can find the following

\textbf{Definition}. \emph{An element} $a$ in a ring $R$ is said to have 
\textsl{left\ stable range 1} (for short lsr1) if whenever $Ra+Rb=R$ for
some $b\in R$, there is an element $y$ such that $a+yb\in U(R)$.

As already mentioned, $a$ has lsr1 iff whenever $xa+b=1$ for some $a,x,b\in
R $, there exists $y\in R$ such that $a+yb\in U(R)$. Equivalently (we can
eliminate $b$), $a$\emph{\ }has lsr1 iff for every $x$ there exists $y$ such
that $a+y(xa-1)$\ is a unit.

A symmetric definition can be given on the right. An element has \textsl{%
stable range 1} if it has both left and right stable range 1.

To simplify the wording, the element $y$ will be called a \textsl{unitizer}
for $a$, depending on $x$.

We abbreviate \emph{the property} "stable range 1" by "sr1" and when useful, 
$sr1(R)$ denotes \emph{the set} of all the both left and right stable range
1 elements in a ring $R$.

An element in a ring is said to be \textsl{clean} if it is a sum of an
idempotent and a unit. A ring is called clean if so are all its elements. A
clean element is called \textsl{nontrivial clean} if the idempotent in its
decomposition is nontrivial (i.e. $\neq 0,1$).

An element $a\in R$ is called (left) \textsl{exchange} if there exists $m\in
R$ such that $a+m(a-a^{2})$ is idempotent.

\emph{Any clean element (or ring) is exchange} (see \cite{nic}), but both
converses fail. Examples of exchange rings which are not clean were given by
G. Bergman (see \cite{han}, Example 1) and by J. \v{S}ter (see \cite{ste},
Example 3.1).

An element $a$ in a ring $R$ is called \textsl{unit-regular} if there exist
a unit $u\in U(R)$ such that $a=aua$.

\bigskip

In this note we deal with $2\times 2$ matrices over commutative rings. Our
results are: we prove a characterization theorem for left stable range 1
elements, and we show that the stable range 1 property is left-right
symmetric (also) at element level. We show that all matrices over B\'{e}zout
rings with one zero row or one zero column have stable range 1.

Using diagonal reduction, we characterize all the $2\times 2$ integral
matrices which have stable range 1, precisely as the matrices whose
determinant is in $\{-1,0,1\}$. This characterization permits to address
some additional properties for stable range 1 elements, like the
"complementary property" and Jacobson's Lemma.

We describe on one example, the hard process of finding unitizers for given
matrices. Finally, we give in $\mathbb{M}_{2}(\mathbb{Z})$, an example of
exchange stable range 1 element that is not clean.

\section{Main results}

In any ring $R$, $0$, all units and all elements of the Jacobson radical
have sr1. As for the latter, since $a\in J(R)$ iff $1-xa$ is left invertible
for any $x\in R$, we choose the unitizer $y=-(1-a)(1-xa)^{-1}$. The only sr1
elements of $\mathbb{Z}$ are the units $\{\pm 1\}$ and zero.

Next, a useful

\begin{lemma}
\label{unu}(i) If $a$ has lsr1 and $u\in U(R)$ then $ua$ has lsr1.

(ii) If $a$ has lsr1, so is $-a$.

(iii) Left sr1 elements are invariant to conjugations.

(iv) If $a$ has lsr1 and $u\in U(R)$ then $au$ has lsr1.

(v) Left sr1 elements are invariant to equivalences.
\end{lemma}

\begin{proof}
(i) Suppose $x(ua)+b=1$. There is $y$ such that $a+yb\in U(R)$. By left
multiplication with $u$ we get $ua+uyb\in U(R)$, as desired.

(ii) Just take $u=-1$ in (i).

(iii) For every $x$ there is a $y$ such that $a+y(xa-1)\in U(R)$. Then $%
u^{-1}[a+y(xa-1)]u\in U(R)$ but we can write this as $%
u^{-1}au+u^{-1}yu[(u^{-1}xu)(u^{-1}au)-1]$, as desired.

(iv) If $a$ has lsr1 and $u\in U(R)$ then $u^{-1}au$ has lsr1, by (ii). Then
by (i), $u(u^{-1}au)=au$ has lsr1.

(v) Follows from (i) and (iii).
\end{proof}

A symmetric statement holds for right sr1 elements.

\bigskip

A result which supersedes all properties in the previous lemma (see \cite%
{che1}, Lemma 17) is the following

\begin{proposition}
\label{lam}Any finite product of left (or right) stable range 1 elements has
left (or right) stable range 1.
\end{proposition}

Since this result simplifies a lot, some of our proofs, we shall use it
subsequently.

\bigskip

Working with square matrices, $A\in \mathbb{M}_{2}(R)$ has left stable range
1 iff whenever $XA+B=I_{2}$ there exists $Y\in \mathbb{M}_{2}(R)$ such that $%
A+YB$ is a unit.

Equivalently, $A\in \mathbb{M}_{2}(R)$ has left stable range 1 iff for every 
$X\in \mathbb{M}_{2}(R)$ there is (a unitizer) $Y\in \mathbb{M}_{2}(R)$ such
that $A+Y(XA-I_{2})$ is invertible.

In the sequel, we use the notation $\mathrm{diag}(r,s):=\left[ 
\begin{array}{cc}
r & 0 \\ 
0 & s%
\end{array}%
\right] $.

Next, we record another useful

\begin{lemma}
\label{trans}(i) A matrix $A$\ has left sr1 iff the transpose $A^{T}$\ has
right sr1.

(ii) $\mathrm{diag}(r,s)$ has left (or right) sr1 iff $\mathrm{diag}(s,r)$
has left (respectively right) sr1.

(iii) $\mathrm{diag}(r,s)$ has left (or right) sr1 iff $\mathrm{diag}(r,-s)$
has left (respectively right) sr1.
\end{lemma}

\begin{proof}
(i) Indeed, $(A+Y(XA-I_{2}))^{T}=A^{T}+(A^{T}X^{T}-I_{2})Y^{T}$ is also a
unit.

(ii) Follows from (iii), the previous lemma, by conjugation with the
involution $E_{12}+E_{21}$.

(iii) Follows from (v), the previous lemma, since $\mathrm{diag}(r,-s)$ is
equivalent to $\mathrm{diag}(r,s)$.
\end{proof}

\textbf{Remark}. When dealing with integral diagonal matrices $\mathrm{diag}%
(n,m)$, with respect to left (or right) sr1, we can suppose $0\leq n\leq m$.

\bigskip

The case of diagonal matrices is of utmost importance because of the
following

\textbf{Definition}. Let $R$ be a commutative unital ring. An $n\times n$
matrix $A$ has a \textsl{diagonal reduction} if there exist units $U$, $V$
such that $UAV=\mathrm{diag}(d_{1},d_{2},...,d_{n})$ is a diagonal matrix,
such that $d_{i}$ divides $d_{i+1}$ for every $1\leq i\leq n-1$.

Following Kaplansky, a ring $R$ is called an \textsl{elementary divisor}
ring if every matrix admits a diagonal reduction. Any diagonal reduction of $%
A$ is called the \textsl{Smith normal form} of $A$. Every PIR (principal
ideal ring) is an elementary divisor ring (see \cite{bro}) and in
particular, $\mathbb{Z}$ is an elementary divisor ring.

\bigskip

In Lemma \ref{unu}, we saw that having stable range 1 property is invariant
to equivalences. Hence

\begin{proposition}
\label{ele}Let $R$ be an elementary divisor ring and $A\in \mathbb{M}_{n}(R)$%
. Then $A$ has stable range 1 iff the Smith normal form of $A$ has stable
range 1.
\end{proposition}

Therefore, over an elementary divisor ring, the determination of the
matrices which have sr1, reduces to diagonal matrices.

\bigskip

As our first main result, we characterize the $2\times 2$ left stable range
1 matrices over any commutative ring.

\begin{theorem}
\label{lsr1}Let $R$ be a commutative ring and $A\in \mathbb{M}_{2}(R)$. Then 
$A$ has left stable range 1 iff for any $X\in \mathbb{M}_{2}(R)$ there
exists $Y\in \mathbb{M}_{2}(R)$ such that 
\begin{equation*}
\det (Y)(\det (X)\det (A)-\mathrm{Tr}(XA)+1)+\det (A(\mathrm{Tr}(XY)+1))-%
\mathrm{Tr}(A\mathrm{adj}(Y))
\end{equation*}%
is a unit of $R$. Here $\mathrm{adj}(Y)$ is the classical adjoint.
\end{theorem}

\begin{proof}
As already mentioned, $A$ has lsr1 in $\mathbb{M}_{2}(R)$ iff for every $X=%
\left[ 
\begin{array}{cc}
a & b \\ 
c & d%
\end{array}%
\right] \in \mathbb{M}_{2}(R)$ there is $Y=\left[ 
\begin{array}{cc}
x & y \\ 
z & t%
\end{array}%
\right] \in \mathbb{M}_{2}(R)$ such that $A+Y(XA-I_{2})$ is invertible.
Since the base ring is supposed to be commutative, $A+Y(XA-I_{2})$ is
invertible in $\mathbb{M}_{2}(R)$ iff $\det (A+Y(XA-I_{2}))$ is a unit of $R$%
. For $A=\left[ 
\begin{array}{cc}
a_{11} & a_{12} \\ 
a_{21} & a_{22}%
\end{array}%
\right] $, the computation amounts to the determinant of the $2\times 2$
matrix with columns 
\begin{equation*}
\begin{array}{c}
C_{1}=\left[ 
\begin{array}{c}
a_{11}+(aa_{11}+ba_{21}-1)x+(ca_{11}+da_{21})y \\ 
a_{21}+(aa_{11}+ba_{21}-1)z+(ca_{11}+da_{21})t%
\end{array}%
\right] \\ 
\mathrm{and} \\ 
C_{2}=\left[ 
\begin{array}{c}
a_{12}+(aa_{12}+ba_{22})x+(ca_{12}+da_{22}-1)y \\ 
a_{22}+(aa_{12}+ba_{22})z+(ca_{12}+da_{22}-1)t%
\end{array}%
\right] .%
\end{array}%
\end{equation*}%
In computing this determinant, there are several terms we gather as follows:

the coefficient of $xz$: $%
(aa_{11}+ba_{21}-1)(aa_{12}+ba_{22})-(aa_{11}+ba_{21}-1)(aa_{12}+ba_{22})$,
which equals zero,

the coefficient of $xt$: $%
(aa_{11}+ba_{21}-1)(ca_{12}+da_{22}-1)-(ca_{11}+da_{21})(aa_{12}+ba_{22})=%
\det (X)\det (A)-aa_{11}-ba_{21}-ca_{12}-da_{22}+1$

the coefficient of $yz$: $%
(ca_{11}+da_{21})(aa_{12}+ba_{22})-(aa_{11}+ba_{21}-1)(ca_{12}+da_{22}-1)=-%
\det (X)\det (A)+aa_{11}+ba_{21}+ca_{12}+da_{22}-1$

the coefficient of $yt$: $%
(ca_{11}+da_{21})(ca_{12}+da_{22}-1)-(ca_{11}+da_{21})(ca_{12}+da_{22}-1)$,
which equals zero,

and another five terms

$a_{11}[(aa_{12}+ba_{22})z+(ca_{12}+da_{22}-1)t]$, $%
a_{22}[(aa_{11}+ba_{21}-1)x+(ca_{11}+da_{21})y]$,

$-a_{12}[(aa_{11}+ba_{21}-1)z+(ca_{11}+da_{21})t]$, $%
-a_{21}[(aa_{12}+ba_{22})x+(ca_{12}+da_{22}-1)y]$,

$\det (A)$.

Then this determinant is 
\begin{equation*}
\begin{array}{c}
\det (Y)(\det (X)\det (A)-aa_{11}-ba_{21}-ca_{12}-da_{22}+1)+ \\ 
+a_{11}[(aa_{12}+ba_{22})z+(ca_{12}+da_{22}-1)t]+ \\ 
+a_{22}[(aa_{11}+ba_{21}-1)x+(ca_{11}+da_{21})y]- \\ 
-a_{12}[(aa_{11}+ba_{21}-1)z+(ca_{11}+da_{21})t]- \\ 
-a_{21}[(aa_{12}+ba_{22})x+(ca_{12}+da_{22}-1)y]+\det (A)%
\end{array}%
\end{equation*}%
or%
\begin{equation*}
\begin{array}{c}
\det (Y)(\det (X)\det (A)-aa_{11}-ba_{21}-ca_{12}-da_{22}+1)+ \\ 
+\det (A)(ax+bz+cy+dt+1)-a_{11}t+a_{12}z+a_{21}y-a_{22}x%
\end{array}%
.
\end{equation*}

Finally this gives the condition in the statement.
\end{proof}

\begin{corollary}
Let $R$ be a commutative ring and $A\in \mathbb{M}_{2}(R)$. Then $A$ has
left stable range 1 iff $A$ has right stable range 1.
\end{corollary}

\begin{proof}
Using the properties of determinants, the properties of the trace and the
commutativity of the base ring, it is readily seen that changing $A,X,Y$
into transposes and reversing the order of the products does not change the
condition in the previous theorem.
\end{proof}

Using this characterization, some special cases are worth mentioning. Since
in the sequel, the base ring for the matrices we consider is commutative,
according to the previous corollary, we will drop the "left" or "right" word
before $2\times 2$ stable range 1 matrices.

\begin{proposition}
\label{trei}Let $R$ be any ring.

(a) Idempotents have sr1.

(b) In $\mathbb{M}_{n}(R)$, $n\geq 2$, all matrices $rE_{ij}$ with $1\leq
i,j\leq n$ and $r\in R$, have sr1.
\end{proposition}

\begin{proof}
(a) Idempotents are unit-regular and unit-regular elements have sr1 (see
Theorem 3.2 in \cite{kl}).

(b) Since sr1 is invariant to equivalences, using two permutation matrices
(interchange first and $i$-th row, interchange of first and $j$-th column),
it suffices to show that $rE_{11}$ has sr1. For every $n\times n$ matrix $X$
we indicate a unitizer $Y$ such that $rE_{11}+Y(X(rE_{11})-I_{n})$ has
determinant $(-1)^{n}$.

Take $Y=\left[ 
\begin{array}{ccccc}
0 & 0 & \cdots & 0 & 1 \\ 
0 & 0 & \cdots & 1 & 0 \\ 
\vdots & \vdots & \ddots & \vdots & \vdots \\ 
0 & 1 & \cdots & 0 & 0 \\ 
1 & 0 & \cdots & 0 & (-1)^{n}a_{1}%
\end{array}%
\right] $, where $\mathrm{col}_{1}(X)=\left[ 
\begin{array}{c}
a_{1} \\ 
\vdots \\ 
a_{n}%
\end{array}%
\right] $. Then we obtain $rE_{11}+Y(X(rE_{11})-I_{n})=\left[ 
\begin{array}{ccccc}
r(1+a_{n}) & 0 & \cdots & 0 & -1 \\ 
ra_{n-1} & 0 & \cdots & -1 & 0 \\ 
\vdots & \vdots & \ddots & \vdots & \vdots \\ 
ra_{2} & -1 & \cdots & 0 & 0 \\ 
ra_{1}(1+a_{n})-1 & 0 & \cdots & 0 & (-1)^{n}a_{1}%
\end{array}%
\right] $. A simple computation shows that $\det
[rE_{11}+Y(X(rE_{11})-I_{n})]=(-1)^{n}$, as desired.
\end{proof}

\textbf{Remarks}. 1) The property (b) above is an Exercise in \cite{lam1}
(Section 24, Exercise 19, (3)\ A). The above proof is our solution.

2) For the $2\times 2$ cases, $rE_{11}$, $rE_{12}$, $rE_{21}$, $rE_{22}$ and 
$X=\left[ 
\begin{array}{cc}
a & b \\ 
c & d%
\end{array}%
\right] $, unitizers are $\left[ 
\begin{array}{cc}
0 & 1 \\ 
1 & a%
\end{array}%
\right] $, $\left[ 
\begin{array}{cc}
1 & 0 \\ 
c & 1%
\end{array}%
\right] $, $\left[ 
\begin{array}{cc}
1 & b \\ 
0 & 1%
\end{array}%
\right] $ and $\left[ 
\begin{array}{cc}
d & 1 \\ 
1 & 0%
\end{array}%
\right] $, respectively.

\begin{corollary}
Let $R$ be any ring. In $\mathbb{M}_{2}(R)$, units, idempotents, and
matrices with three zeros, all have sr1.
\end{corollary}

Recall that a commutative ring is called \textsl{B\'{e}zout} if any two
elements have a greatest common divisor that is a linear combination of them.

Our second main result is the following

\begin{theorem}
\label{nm}Let $R$ be any B\'{e}zout ring. All matrices in $\mathbb{M}_{2}(R)$
with (at least) one zero row or zero column have sr1.
\end{theorem}

\begin{proof}
By Lemma \ref{trans}, it suffices to prove the claim for matrices with zero
second row.

Let $A=\left[ 
\begin{array}{cc}
r & s \\ 
0 & 0%
\end{array}%
\right] $ with $r,s\in R$. Replacement in the characterization theorem ($%
a_{11}=r$, $a_{12}=s$, $\det (A)=0$) amounts to 
\begin{equation*}
\det (Y)(1-ra-sc)-rt+sz=\pm 1.
\end{equation*}%
We go into two cases: (i) $\gcd (r;s)=1$, and (ii) $\gcd (r;s)\neq 1$.

\textbf{(i)} Since $r,s$ are coprime, $z,t$ can be chosen for $-rt+sz=1$
(say $z_{0},t_{0})$. Choosing $x=y=0$ (and so $\det (Y)=0$) we get a
unitizer of form $Y=\left[ 
\begin{array}{cc}
0 & 0 \\ 
z_{0} & t_{0}%
\end{array}%
\right] $ (which is independent of $a,b,c,d$).

\textbf{(ii)} In the remaining case, suppose $1\neq d$ and let $r,s$ be
coprime. Then $A=\left[ 
\begin{array}{cc}
dr & ds \\ 
0 & 0%
\end{array}%
\right] =\left[ 
\begin{array}{cc}
d & 0 \\ 
0 & 0%
\end{array}%
\right] \left[ 
\begin{array}{cc}
r & s \\ 
0 & 0%
\end{array}%
\right] $ is sr1, as product of sr1 matrices (by Proposition \ref{trei} (b)
and (i), this theorem).
\end{proof}

\textbf{Remarks}. 1) Matrices of the form $\left[ 
\begin{array}{cc}
a & ab \\ 
0 & 0%
\end{array}%
\right] $ are sr1 over any (possibly not commutative) ring $R$. To see this,
we decompose $\left[ 
\begin{array}{cc}
a & ab \\ 
0 & 0%
\end{array}%
\right] =\left[ 
\begin{array}{cc}
a & 0 \\ 
0 & 0%
\end{array}%
\right] \left[ 
\begin{array}{cc}
1 & b \\ 
0 & 0%
\end{array}%
\right] $. Both are sr1 by Proposition \ref{trei} (the right one is
idempotent) and the fact that $sr1(R)$ is multiplicatively closed.

2) When it comes to find a unitizer in case (ii), the difficulties which
occur are described in the next section, on a special case.

3) For $r=s$, one can also use another unitizer: $Y=\left[ 
\begin{array}{cc}
1 & 0 \\ 
a+c+1 & 1%
\end{array}%
\right] $. Similarly, for $r=-s$: $Y=\left[ 
\begin{array}{cc}
1 & 0 \\ 
-a+c-1 & 1%
\end{array}%
\right] $.

4) If $r$, $s$ are co-prime then $\left[ 
\begin{array}{cc}
r & s \\ 
0 & 0%
\end{array}%
\right] $ is unit regular (see \cite{khu}) and thus has stable range one.
This is an alternative proof of case (i) of the previous theorem (not
providing unitizers).

\bigskip

In our third main result, we characterize the integral sr1 matrices. We
first discuss a special case.

\begin{lemma}
\label{di}An integral diagonal matrix $A=nI_{2}$ has sr1 iff $n\in
\{-1,0,1\} $.
\end{lemma}

\begin{proof}
According to the remark after Lemma 3, suppose $1\leq n$. For every multiple
of $I_{2}$, we have to indicate an $X$ for which no $Y$ exists such that $%
A+Y(XA-I_{2})$ has determinant $\pm 1$.

Since for $n=1$, $I_{2}$ is a unit, we take $A=nI_{2}$ for $n\geq 2$ and
consider $X=-(n+1)I_{2}$. Then $Y(XA-I_{2})=-(1+n+n^{2})Y$ and we can
compute the determinant in the ring $\mathbb{Z}/(1+n+n^{2})\mathbb{Z}$. The
characterization becomes $n^{2}$ congruent to ${\pm 1}$ mod $(1+n+n^{2})$,
which is impossible since $n\geq 2$. Hence multiples $nI_{2}$ with $n\geq 2$
have not sr1.
\end{proof}

Since units are known to have sr1, we have the following

\begin{theorem}
\label{sr1}Let $R=\mathbb{M}_{2}(\mathbb{Z})$. Then a matrix $A\in
R\setminus U(R)$ has sr1 iff $\det (A)=0$.
\end{theorem}

\begin{proof}
As mentioned before, since $\mathbb{Z}$ is an elementary divisor ring, by
Proposition \ref{ele}, we may assume that $A$ is diagonal, say $\mathrm{diag}%
(n,m)$, where $m,n\geq 0$. If $\det (A)=0$, then Proposition \ref{trei}, (b)
shows that $A$ has sr1. Conversely, assume that $\det (A)\neq 0$. Then $%
m,n\geq 1$. If $A$ has sr1, by Lemma \ref{trans}, (ii), $\mathrm{diag}(m,n)$
has also sr1, and so is their product, $nmI_{2}$. As $A\notin U(R)$, we have 
$mn\geq 2$ and this is impossible by the previous lemma.
\end{proof}

\begin{corollary}
In $\mathbb{M}_{2}(\mathbb{Z})$ all idempotents and all nilpotents have
stable range 1.
\end{corollary}

\textbf{Remarks}. 1) Naturally, for $R=\mathbb{Z}$, Theorem \ref{nm} follows
since all matrices have zero determinant.

2) Matrices $M_{uv}=\left[ 
\begin{array}{cc}
1 & u \\ 
v & uv%
\end{array}%
\right] $ have sr1 over any commutative ring.

Indeed, having zero determinant, the characterization yields $\det
(Y)(1-a-vb-uc-uvd)-t+uz+vy-uvx=\pm 1$ for which the unitizer

$Y=\left[ 
\begin{array}{cc}
0 & 1 \\ 
-1 & 1-a-vb-uc-uvd%
\end{array}%
\right] $ gives $+1$. Again, for $R=\mathbb{Z}$, this follows since such
matrices have zero determinant.

3) The results in this section yield simple examples which show that $sr1(%
\mathbb{M}_{2}(\mathbb{Z}))$ is \emph{not} closed under addition.

Indeed, $E_{11}$, $I_{2}$ both are sr1 but the (diagonal) sum is not.

4) The previous characterization can be also obtained as follows. As every
matrix over $\mathbb{Z}$ has a Smith normal form and stable range one is
invariant under multiplication by units, it suffices to see which matrices
of the form $\left[ 
\begin{array}{cc}
a & 0 \\ 
0 & b%
\end{array}%
\right] $ over $\mathbb{Z}$ have stable range one. It is not hard to see
that this is the case when either one of the $a$, $b$ is zero or both $a$
and $b$ are units.

\bigskip

Another consequence of the previous theorem is the following

\begin{corollary}
(i) In $\mathbb{M}_{2}(\mathbb{Z})$, stable range 1 elements do not have the
"complementary property".

(ii) In $\mathbb{M}_{2}(\mathbb{Z})$, $AB$ has stable range 1 iff so has $BA$%
.

(iii) Jacobson's Lemma holds for stable range 1 matrices in $\mathbb{M}_{2}(%
\mathbb{Z})$.
\end{corollary}

\begin{proof}
(i) Indeed, in $\mathbb{M}_{2}(\mathbb{Z})$, $-I_{2}$ is a unit, so has sr1.
However, $I_{2}-(-I_{2})=2I_{2}$ has not sr1.

(ii) Suppose $AB$ has sr1, i.e. $\det (AB)\in \{-1,0,1\}$. Then $\det
(BA)\in \{-1,0,1\}$.

(iii) In $\mathbb{M}_{2}(\mathbb{Z})$ we have to verify whether $\det
(I_{2}-AB)\in \{-1,0,1\}$ implies $\det (I_{2}-BA)\in \{-1,0,1\}$. Since $%
\det (I_{2}-M)=1+\det (M)-\mathrm{Tr}(M)$ holds for any $2\times 2$ matrix $%
M $, we deduce $\det (I_{2}-AB)=1+\det (AB)-\mathrm{Tr}(AB)=1+\det
(BA)-Tr(BA)=\det (I_{2}-BA)$ and so the claim follows.
\end{proof}

Jacobson's Lemma holds for units, regular or unit-regular elements, $\pi $%
-regular or strongly $\pi $-regular elements.

Since unit-regular elements have stable range 1, at least \emph{for this
subset}, Jacobson's Lemma generally holds.

\section{Finding unitizers}

In Theorem \ref{nm}, the unitizers were found by observation (with computer
aid). In this section we show that in case (v), finding a unitizer is not so
easy. We focus on $A=\left[ 
\begin{array}{cc}
6 & 10 \\ 
0 & 0%
\end{array}%
\right] =\left[ 
\begin{array}{cc}
2 & 0 \\ 
0 & 0%
\end{array}%
\right] \left[ 
\begin{array}{cc}
3 & 5 \\ 
0 & 0%
\end{array}%
\right] $ which is sr1, as product of sr1 matrices.

For any given $X=\left[ 
\begin{array}{cc}
a & b \\ 
c & d%
\end{array}%
\right] $ we are looking for a unitizer $Y=\left[ 
\begin{array}{cc}
x & y \\ 
z & t%
\end{array}%
\right] $, such that $\det (Y)(1-6a-10c)-6t+10z$ is $1$ or $-1$.

\bigskip

Having some examples in Theorem \ref{nm}, we made by computer some attempts:
to look for unitizers with zero first row, or for unitizers with $x=t=1$ and 
$y=0$. Since these did not (seem to) cover all situations (\textquotedblleft
seem\textquotedblright , because of given bounds for the entries of $X$ and $%
Y$, respectively), we decided to concentrate on unitizers (already used in
some cases) of form $Y=\left[ 
\begin{array}{cc}
0 & \pm 1 \\ 
\mp 1 & t%
\end{array}%
\right] $, that is, with $\det (Y)=1$, where $t$ is to be found depending on 
$a$ and $c$, since computer seemed to cover all situations.

We go into two cases.

\textbf{Case 1}. $z=-1$, $x=0$, $y=1$. The characterization gives $%
1-6a-10c-6t-10=\pm 1$, which we split into two subcases.

(a) $1-6a-10c-6t-10=1$, which we write $3t+5c+3a=-5$. This is a linear
Diophantine equation with three unknowns, solvable since $\gcd (3;5;3)=1$
divides $-5$. The general solution is $t=5+5k-m$, $c=-4-3k$, $a=m$, whence $%
t=5(k+1)-a$ and $\ c\equiv 2$ (mod 3).

(b) $1-6a-10c-6t-10=-1$, which we write $3t+5c+3a=-4$. Again a Diophantine
equation; the general solution is $t=2+5k-m$, $c=-2-3k$, $a=m$, whence $%
t=2+5k-a$ and $c\equiv 1$ (mod 3).

\textbf{Case 2}. $z=1$, $x=0$, $y=-1$. The characterization gives $%
1-6a-10c-6t+10=\pm 1$, which again we split into two subcases.

(i) $1-6a-10c-6t+10=1$, which we write $3t+5c+3a=5$. The general solution is 
$t=5k-m$, $c=1-3k$, $a=m$, whence $t=5k-a$ and $c\equiv 1$ (mod 3).

(ii) $1-6a-10c-6t+10=-1$, which we write $3t+5c+3a=6$. The general solution
is $t=2+5k-m$, $c=-3k$, $a=m$, whence $t=2+5k-a$ and $c\equiv 0$ (mod 3).

Therefore, there are (slightly) different unitizers, corresponding to the
reminder of the division of $c$ by $3$. More precisely,

if $c\equiv 0$ (mod 3) the unitizer is $Y=\left[ 
\begin{array}{cc}
0 & -1 \\ 
1 & 2-a-\dfrac{5}{3}c%
\end{array}%
\right] $,

if $c\equiv 1$ (mod 3) there are two possible unitizers: $Y=\left[ 
\begin{array}{cc}
0 & -1 \\ 
1 & \dfrac{5}{3}-a-\dfrac{5}{3}c%
\end{array}%
\right] $, or $Y=\left[ 
\begin{array}{cc}
0 & 1 \\ 
-1 & -\dfrac{4}{3}-a-\dfrac{5}{3}c%
\end{array}%
\right] $,

if $c\equiv 2$ (mod 3) the unitizer is $Y=\left[ 
\begin{array}{cc}
0 & 1 \\ 
-1 & -\dfrac{5}{3}-a-\dfrac{5}{3}c%
\end{array}%
\right] $.

\section{Exchange stable range 1, $2\times 2$ matrices, may not be clean}

Given an exchange ring, for this to be clean, one needs units in addition to
idempotents. The stable range 1 condition is known to be excellent in
helping to produce units. That is why it is natural to raise the following

\textbf{Question}: \emph{If an exchange ring }$R$\emph{\ has stable range
one, is }$R$\emph{\ necessarily a clean ring?}

In what follows, using computer aid, we show that, \emph{at element level},
this fails for $R=\mathbb{M}_{2}(\mathbb{Z})$. Clearly, the existence of
such examples should not entirely dash our hopes for a positive answer to
this question. This is because, in working with the question, we are under
the stronger assumption that all (not just some) elements of $R$ are
exchange and have stable range one.

\bigskip

We shall use the following (for a proof, see e.g. Theorem 3, \cite{cal2})

\begin{theorem}
A $2\times 2$ integral matrix $A=\left[ 
\begin{array}{cc}
a & b \\ 
c & d%
\end{array}%
\right] $ is nontrivial clean iff the system%
\begin{equation*}
\left\{ 
\begin{array}{c}
x^{2}+x+yz=0\ \ \ \ \ \ \ (1) \\ 
(a-d)x+cy+bz+\det (A)-d=\pm 1\ \ \ \ (\pm 2)%
\end{array}%
\right.
\end{equation*}%
with unknowns $x,y,z$, has at least one solution over $\mathbb{Z}$. If $%
b\neq 0$ and (2) holds, then (1) is equivalent to%
\begin{equation*}
bx^{2}-(a-d)xy-cy^{2}+bx+(d-\det (A)\pm 1)y=0\ \ \ \ (\pm 3).
\end{equation*}
\end{theorem}

Here $E=\left[ 
\begin{array}{cc}
x+1 & y \\ 
z & -x%
\end{array}%
\right] $ is the (nontrivial) idempotent of a clean decomposition of $A$
(i.e., $A-E$ is a unit).

While it is easy to use this characterization of clean matrices, and the
characterization of sr1 matrices, it is hard to check the exchange property
for matrices. Therefore, computer aid was again necessary.

\bigskip

Our example is $A=\left[ 
\begin{array}{cc}
5 & 5 \\ 
7 & 7%
\end{array}%
\right] $.

\textbf{1}. $A$ is exchange. Indeed, $A+M(A-A^{2})=\left[ 
\begin{array}{cc}
5 & 5 \\ 
-4 & -4%
\end{array}%
\right] $ is an idempotent (determinant = zero, trace = 1) for $M=\left[ 
\begin{array}{cc}
0 & 0 \\ 
3 & -2%
\end{array}%
\right] $.

\textbf{2}. $A$ is \emph{not} clean. We use the theorem above: for $a=b=5$, $%
c=d=7$ (and $\det A=0$) the Diophantine equations are

\begin{equation*}
5x^{2}+2xy-7y^{2}+5x+(7\pm 1)y=0\text{ \ \ \ \ }(\pm 3)
\end{equation*}%
and the corresponding linear equations are 
\begin{equation*}
-2x+7y+5z-7=\pm 1\text{ \ \ \ \ }(\pm 2).
\end{equation*}

The equations ($\pm $3) have only the solutions $(0,0)$ and $(-1,0)$. None
verifies ($\pm $2), so $A$ is not clean.

\textbf{3}. $A$ has stable range 1, by Theorem \ref{sr1}, since $\det A=0$.

\textbf{Remarks}. 1) Using different references and results, another example
can be provided. In \cite{khu}, the matrix $\left[ 
\begin{array}{cc}
12 & 5 \\ 
0 & 0%
\end{array}%
\right] $ is given as an example of unit-regular matrix which is not clean.
Therefore, this example suits well, if we mention that \emph{unit-regular
elements are exchange and sr1} (see \cite{goo} and \cite{kl}).

2) Incidentally, these two examples are similar: if we take $U=\left[ 
\begin{array}{cc}
3 & -2 \\ 
-7 & 5%
\end{array}%
\right] $ then $UAU^{-1}=\left[ 
\begin{array}{cc}
12 & 5 \\ 
0 & 0%
\end{array}%
\right] $.

\bigskip

In closing, just to have an idea of the "density" of such examples, out of
1988 sr1 matrices with entries bounded in absolute value by 9, \emph{all
were also exchange} and 80 were not clean (verification made with entries
bounded in absolute value by 6).

This suggested to ask the following

\bigskip

\textbf{Question}. \emph{Are all }$2\times 2$\emph{\ integral stable range
one matrices, exchange}?

\bigskip

Using the results in our paper, we can give a negative answer.

Indeed, by Theorem \ref{sr1}, we have to check that units are exchange,
which is true since units are clean and so exchange, and, that zero
determinant matrices are exchange. However this fails.

Indeed, one can show that the matrices $\left[ 
\begin{array}{cc}
2k+1 & 0 \\ 
0 & 0%
\end{array}%
\right] $ are \emph{not} exchange for any $k\notin \{-1,0\}$.

\begin{acknowledgement}
Thanks are due to Tsit Yuen Lam for suggestions, completions and his kind
permission to use material from the forthcoming \cite{lam1} and to the
referee whose corrections and comments improved our presentation.
\end{acknowledgement}

\bigskip

G. C\u{a}lug\u{a}reanu,

Department of Mathematics, Babe\c{s}-Bolyai University,

1 Kog\u{a}lniceanu Street, Cluj-Napoca, Romania

E-mail: calu@math.ubbcluj.ro

\bigskip

H. F. Pop,

Department of Computer Science, Babe\c{s}-Bolyai University,

1 Kog\u{a}lniceanu Street, Cluj-Napoca, Romania

E-mail: hfpop@cs.ubbcluj.ro

\end{document}